\documentclass{article}[1100]
\usepackage{amssymb}
\usepackage{amsmath}
\usepackage{amsthm}
\usepackage{graphicx}
\usepackage[cal=cm]{mathalfa}
\usepackage{colonequals}
\usepackage[
backend=bibtex,
style=numeric,
sorting=nyt
]{biblatex}
\newtheorem{theorem}{Theorem}[section]
\newtheorem{lemma}[theorem]{Lemma}
\newtheorem{proposition}[theorem]{Proposition}

\theoremstyle{definition}

\theoremstyle{remark}
\newtheorem{remark}[theorem]{Remark}

\usepackage{tcolorbox}
\newcommand\eps{\varepsilon}
\newcommand\sgn{\text{sgn}}
\newcommand\innerproduct[2]{\langle #1,#2\rangle}
\newcommand\Ren[1]{\mathbb{R}^{#1}}
\newcommand\ar{\mathbb{R}}
\newcommand\norma[1]{|\!|#1|\!|}

\newcommand\inverseRadonPart[2]{\left(\frac{1}{{#1}}\frac{d}{d{#1}}\right)^{#2}}

\newcommand\testFunctionSpace{C_e^\infty(\sphere{n-1})}
\newcommand\distributionSpace{D_e(\sphere{n-1})}

\newcommand\mySpaceH[1]{\textsc{H}(#1)}
\newcommand\mySpace[2]{\textsc{H}_{#1}(#2)}
\newcommand\sphere[1]{S^{#1}}

\newcommand\directSum[1]{\ar\oplus_{\scriptscriptstyle E}\mkern-3mu#1}
\newcommand{\littleE}{\scriptscriptstyle E}
\newcommand{\CT}{\mathcal T}
\newcommand{\inverseCT}{{\mathcal T}^{-1}}
\newcommand\smallcirc{\scriptstyle\circ}
\title{On rotationally-symmetric norms}
\author{Yossi Lonke}

\begin{document}
\maketitle





\begin{abstract}
For two-dimensional norms that generate rotationally-symmetric norms in all higher dimensions, necessary and
sufficient conditions are established for them to be subspaces of $L_1$. 
\end{abstract}

\maketitle

\section*{Introduction}
The aim of the present paper
is to characterize rotationally-symmetric norms in $\Ren{n}$ for which the corresponding normed space is an $L_1$-subspace, that is,  isomorphically  isometric to a subspace of $L_1(0,1)$. 
A study of such norms naturally arises from the work of Dor (\cite{Dor}). See Section 1 for a detalied discussion.

\medskip
We say that a function $f\colon\Ren{n}\to\ar$ is {$v$-rotationally symmetric} for some unit vector $v$ in $\Ren{n}$,  if
 for every orthogonal transformation $U$ that keeps $v$ fixed,  $f\circ U=f$. A function is \emph{rotationally symmetric}
 if it is $v$-rotationally symmetric for some unit vector~$v$.
 
 \medskip
Let $\norma{\cdot}$ be an $e_1$-rotationally symmetric norm in $\Ren{n+1}$. If $E$ is the two dimensional
subspace  spanned by $e_1$ and $e_2$, then rotational symmetry implies that $\{e_1 ,e_2\}$ is a \emph{sign-symmetric basis} for $E$.  This means that for every choice of  scalars $a_1$ and~$a_2$, the vectors $-a_1e_1+a_2e_2$ and
$a_1e_1+a_2e_2$ have the same norm. Conversely, with every  two-dimensional, real normed space $E$  that has a sign-symmetric basis $\{f_1,f_2\}$, we associate a rotationally-symmetric norm in $\Ren{n+1}$ by
\begin{equation}\label{normDefinition}
\norma{(t,a_1,\dots,a_n)}=\norma{(\sum_{i=1}^na_i^2)^{1/2}f_1+tf_2}_E.
\end{equation}
An infinite dimensional analogue of (\ref{normDefinition}) arises if instead of $(a_i)_{i=1}^n$ we take a sequence
$(a_n)_{n=1}^{\infty}$ in $\ell_2$. 

\medskip

The function $\varphi(t)=\norma{f_1+tf_2}_{\scriptstyle{E}}$ with $t\in\ar$ plays a central role
in our characterization.  The main result of Section~$2$ is that if $\varphi''(t)$ is a continuous function, then  $\Ren{n+1}$
equipped with the norm (\ref{normDefinition})  is an $L_1$-subspace for every~$n\geq 2$ if and only if $\varphi''(\sqrt{t})$ is completely monotone in $[0,\infty)$. This result encapsulates  infinitely many derivatives of $\varphi$, and as will become clear, it is
 essentially a statement about the infinite-dimensional analogue of (\ref{normDefinition}). Moreover, it suggests that  derivatives of a fixed order, e,g. $\frac{d^k}{dt^k}(\varphi''(\sqrt{t}))$,
are related to the finite-dimensional case. This is indeed the case. In Section~$3$ we prove that for $n\geq 2$,  if a rotationally-symmetric norm in $\Ren{2n}$
produces an $L_1$-subspace, then $(-1)^{n-1}\frac{d^{n-1}}{dt^{n-1}}(\varphi''(\sqrt{t}))$ is a positive measure on $(0,\infty)$. In particular, the function $\varphi$ must have at least $n-1$ continuous derivatives.

 \medskip

Section~$4$ is devoted to an important example, that generates an infinite dimensional $L_1$-subspace of the form 
$\mathbb{G}_1\oplus \ar$, where $\mathbb{G}_1$ is a Gaussian Hilbert space embedded in $L_1$. A two dimensional calculation yields
an explicit expression for the dual norm.  Previously, I had conjectured that the dual space of $\mathbb{G}_1\oplus \ar$, is also an $L_1$-subspace. (\cite{KoldobskyPrivate}, \cite{SchneiderPrivate}).
A Contemporaneous  result of  \cite{RyaZvav}  confirms the conjecture. This provides a negative answer to a question of Grothendieck and a problem of Schneider.  
\medskip

\medskip

All vector spaces in this work are over the field of real numbers.

\section{basic properties of rotationally symmetric norms}

The norm  in (\ref{normDefinition}) is defined in terms of the standard coordinates in $\Ren{n+1}$, but it depends on a choice of
the basis $\{f_1,f_2\}$ for~$E$. The space $\ell_1^2$, for example,  has two different sign-symmetric bases: $\{(1,0),(0,1)\}$ and also $\{(1,1),(-1,1)\}$. If $n\geq 2$, then
with these two bases, the equation  (\ref{normDefinition}) produces two non isometric norms in $\Ren{n+1}$:\
\[ (t,a_1,\dots,a_n)\to (\sum_{i=1}^na_i^2)^{1/2}+|t|,\ \ \mbox{and} \ \ (t,a_1,\dots,a_n)\to 2\max\{(\sum_{i=1}^na_i^2)^{1/2}, |t|\}. \]

Therefore, if $E$ is a two-dimensional normed space with a sign-symmetric basis~$\beta=\{f_1,f_2\}$, then we denote by  $\mySpace{n+1}{E,\beta}$  
 the space $\Ren{n+1}$ equipped with the norm (\ref{normDefinition}), and by   $\mySpaceH{E,\beta}$ the space $\ar\oplus\ell_2$ equipped with  the 
 infinite-dimensional analogue of the norm (\ref{normDefinition}).
 
 \medskip

We begin with a simple proof of  a monotonicity property of sign-symmetric bases. It is essential for verifying
the triangle inequality of the norm defined by~(\ref{normDefinition}), and will be used several times in the following sections.

\begin{tcolorbox}
\begin{lemma}\label{easyLemma}

If $(E,\norma{\cdot})$ is a two-dimensional normed space with a sign-symmetric basis  $\{f_1,f_2\}$, then
\[{\rm (i)}\quad \hbox{For every $y$, if $|s|\leq |t|$ then $\norma{sf_1+yf_2}\leq \norma{tf_1+yf_2}$}, \]
and
\[{\rm (ii)}\quad \hbox{For every $x$, if $|s|\leq |t|$ then $\norma{xf_1+sf_2}\leq \norma{xf_1+tf_2}$}.\]

\end{lemma}
\end{tcolorbox}
\begin{proof}
We prove (i). The proof of (ii) is similar.

\medskip

Put $\nu(x,y)=\norma{xf_1+yf_2}$.
Fix $y$, and let $s,t$ be such that $|s|\leq |t|$. Convexity of the norm combined with sign-symmetry implies, for every~$x$,
\begin{equation}\label{fconvex}
\nu(0,y)\leq\frac{1}{2}(\nu(-x,y)+\nu(x,y))=\nu(x,y).
\end{equation}
Pick a point $p=(t,y)$, with $t>0$. Sign-symmetry allows us to assume that $0\leq s<t$. The point $p$ belongs
to the convex set $B=\{q\in\Ren{2}: \norma{q}\leq\norma{p}\}$, and by (\ref{fconvex}), so does $(0,y)$. Hence the line segment joining the point $p$ to $(0,y)$ also belongs to the set $B$, so for every $0\leq\lambda\leq 1$, the point
$(\lambda t,y)$ belongs to $B$ as well. In particular, (take $\lambda=s/t$),
the point $(s,y)$ belongs to $B$. Hence,
\[ \nu(s,y)\leq \nu(p)=\nu(t,y), \]
and (i) is proved.
\end{proof}


If $X$ is a Banach space isomorphic to a Hilbert space, then  by $d(X)$ we denote the Banach-Mazur distance between $X$ and the
Hilbert space of its dimension. Our proof of the next proposition relies on the fact that the distance-ellipse of a two-dimensional
normed space is unique, up to homothety. An unpublished theorem by Maurey, mentioned in  (\cite{AFJS}, Remark 1.2), states that if a finite dimensional normed space does not have a unique distance ellipsoid, then it has a proper subspace whose distance to a Hilbert space is the same as that of the whole space. As a result, the distance ellipse of every two-dimensional normed space is unique, up to homothety.  A  proof of this uniqueness property has recently appeared in (\cite{GrundKobos}, Corollary $2.11$).

\begin{tcolorbox}
\begin{proposition}\label{distanceProposition} Let $E$ be a two dimensional normed space  with a  sign-symmetric basis 
${\beta=\{f_1,f_2\}}$.
Put
\[ \textsc{H}\colonequals\mySpaceH{E,\beta} \ \  \mbox{and} \ \  \textsc{H}_{n+1}\colonequals\mySpace{n+1}{E,\beta}.\] 
The space $\textsc{H}$ is isomorphic to $\ell_2$, and for every $n\geq 1$,
\begin{equation}\label{equalDistances}
d(\textsc{H}_{n+1})= d(\textsc{H})=d(E).
\end{equation}
\end{proposition}
\end{tcolorbox}
\begin{proof}
One direction of inequalities follows from the inclusions
\[ E\cong\textsc{H}_2\subset\textsc{H}_{n+1}\subset\textsc{H},\] 
valid for all $n\geq 1$. Hence
\[ d(E)\leq d(\textsc{H}_{n+1})\leq d(\textsc{H}).\]
To complete the proof, it remains to show that $\textsc{H}$ is isomorphic to $\ell_2$ and that ${d(\textsc{H})\leq d(E)}$.

\medskip

If $B_{\scriptscriptstyle E}$ is the unit ball of $E$ and $B_2$ the unit ball of $\ell_2^2$, then there exists a linear isomorphism ${L:\Ren{2}\to E}$  such that
\begin{equation}\label{distanceEllipsoidInclusion}
B_{\scriptscriptstyle E}\subset L(B_2)\subset d(E)B_{\scriptscriptstyle E},
\end{equation}
where $L(B_2)$ is the distance ellipse of $E$. Put $u_1=L^{-1}f_1$ and $u_2=L^{-1}f_2$. A vector $x_1f_1+x_2f_2$ belongs to $L(B_2)$ if and only if $x_1u_1+x_2u_2\in B_2$, that is, if and only if in $\Ren{2}$ the point $(x_1,x_2)$
belongs to some linear image of the ball $B_2$. Therefore, there exist real numbers $a,b,c$ such that $a$ and $c$ are nonzero and
\[ x_1f_1+x_2f_2\in\hbox{boundary of  $L(B_2)$}\enspace\hbox{if and only if}\enspace a^2x_1^2+b\,x_1x_2+c^2x_2^2=1. \]

By the uniqueness of the distance ellipse, $L(B_2)$  is invariant under the  sign-symmetries of~$E$.  Hence $b=0$.  
Thus, by (\ref{distanceEllipsoidInclusion}), for every choice of real scalars $x_1,x_2$,
\begin{equation}\label{distance}
(a^2x_1^2+c^2x_2^2)^{1/2}\leq \norma{x_1 f_1+x_2 f_2}\leq d(E) (a^2x_1^2+c^2x_2^2)^{1/2}.
\end{equation}

Define a map $T:\textsc{H}\to\ell_2$ by
\[ T(t,\eta_1,\eta_2,\dots)=(c t,a \eta_1,a \eta_2,\dots).\]
$T$ is clearly an isomorphism onto, and by (\ref{distance}),
\[
\begin{aligned}
\norma{T(t,\eta_1,\eta_2,\dots)}_{\ell_2}&=(c^{2}t^2+a^{2}\sum_{i=1}^{\infty}\eta_i^2)^{1/2}\leq  \norma{(\sum_{i=1}^{\infty}\eta_i^2)^{1/2}f_1+tf_2}\\
&=\norma{(t,\eta_1,\eta_2,\dots)}_{\textsc{H}}. \\
\end{aligned}
\]
Hence, $\norma{T}\leq 1$. The inverse $T^{-1}$ maps a vector $\{x_k\}_{k=1}^{\infty}\in\ell_2$ to the vector
\[(\frac{x_1}{c},\frac{x_2}{a},\frac{x_3}{a},\dots), \]
whose norm in $\textsc{H}$, according to (\ref{distance}), satisfies:
\[\norma{\frac{1}{a}\bigg(\sum_{k=2}^{\infty}x_k^2\bigg)^{1/2}f_1+\frac{x_1}{c}f_2}\leq d(E)\bigg(\sum_{k=1}^{\infty}x_k^2\bigg)^{1/2},\]
whence $\norma{T^{-1}}\leq d(E)$.
Therefore $\textsc{H}$ is  isomorphic to $\ell_2$, and  $d(\textsc{H})\leq d(E)$. 

\end{proof}

The spaces $\mySpaceH{E,\beta}$ are all \emph{isomorphic} to a Hilbert Space, but not \emph{isometric} to a Hilbert Space unless $E$ itself is. The next proposition verifies that the dual space
of $\mySpaceH{E,\beta}$ behaves as expected. The proposition is in fact true in a more general case, but for our purposes it will be convenient
to specialize the proof to the case of a Hilbert space. 
\begin{tcolorbox}
\begin{proposition}\label{dualSpaceProposition}
 Let $E$ be a two dimensional normed space  with a  sign-symmetric basis $\beta=\{f_1,f_2\}$.
 
 \medskip 
 
If $E^*$ is the dual space of $E$, then the algebraic dual basis~$\beta^*=\{f_1^*,f_2^*\}$ is also sign-symmetric 

\medskip

Moreover, for every ${n\geq 1}$
the dual space of $\mySpace{n+1}{E,\beta}$ is isometrically isomorphic to $\mySpace{n+1}{E^*,\beta^*}$ and the dual space of $\mySpaceH{E,\beta}$ is isometrically isomorphic to $\mySpaceH{E^*,\beta^*}$.
\end{proposition}
\end{tcolorbox}
\begin{proof}
For every choice of signs $\eps_1=\pm 1,\eps_2=\pm 1$, the map 
\[ T(x_1f_1+x_2f_2)=\eps_1x_1f_1+\eps_2x_2f_2 \]
is a linear isometry of $E$ onto itself, hence for every choice of scalars $a_1,a_2$ and every $\eps_i=\pm 1$
\[
\norma{a_1\eps_1f_1^*+a_2\eps_2f_2^*}=\norma{a_1f_1^*+a_2 f_2^*}.
\]
Thus $\{f_1^*,f_2^*\}$ is a sign-symmetric basis for the dual space $E^*$.
\vskip 6pt
We prove the infinite-dimensional case. The finite-dimensional case has a similar proof. To avoid
cluttered notation, we  temporarily replace the symbol $\mySpaceH{E,\beta}$ by $\mySpaceH{E}$ and $\mySpaceH{E^*,\beta^*}$ by $\mySpaceH{E^*}$.  Elements of the underlying vector-space $\ar\oplus\ell_2$ are denoted by pairs $(t,v)$, and the standard inner product in $\ell_2$ is denoted by $\langle\cdot,\cdot\rangle$.

\medskip
 
  For each element $(s,v)\in\mySpaceH{E^*}$, let $\Lambda_{s,v}$ denote the linear functional defined on $\mySpaceH{E}$ by
\[ \Lambda_{s,v}(t,u)=st+\innerproduct{v}{u}.\]
 Each $\Lambda_{s,v}$ is bounded, hence the map $(s,v)\to\Lambda_{s,v}$ maps
$\mySpaceH{E^*}$ into $(\mySpaceH{E})^*$. It is clearly linear. It is also surjective, because if $e_0$ denotes the vector in $\mySpaceH{E}$ whose first coordinate is $1$ and all the rest are zero, then it is straightforward to verify that every linear functional $\Lambda\in(\mySpaceH{E})^*$ has the form $\Lambda=\Lambda_{s,v}$ for $s=\Lambda(e_0)$ and a uniquely determined vector $v\in\ell_2$.

\medskip

To prove that our map is an isometry, we must  prove that for every $s\in\ar$ and every $v\in\ell_2$,
\begin{equation}\label{isometryToProve}
\norma{\Lambda_{s,v}}=\norma{(s,v)}_{\mySpaceH{E^*}} .
\end{equation}

\medskip

Fix a pair $(s,v)\in\mySpaceH{E}$. Let $e_1$ denote first vector  of the standard basis of~$\ell_2$. For each $v\in\ell_2$,  let $|v|$ denote its norm. 

\medskip

Now we use the fact that the group of isometries of a Hilbert Space acts transitively on the unit sphere.
Pick an isometry $U$ of $\ell_2$ such that $Uv=|v|e_1$. Then,
\[\begin{aligned}
\norma{\Lambda_{s,v}}&=\sup\{|st+\innerproduct{u}{v}|: (t,u)\in\mySpaceH{E},\ \ \norma{t f_1+|u|f_2}=1\}\\
&=\sup\{|st+|v|\innerproduct{Uu}{e_1}|:(t,u)\in\mySpaceH{E},\ \  \norma{t f_1+|u|f_2}=1\}\\
&\leq \sup\{|st+|v|\innerproduct{Uu}{e_1}|:(t,u)\in\mySpaceH{E},\ \  \norma{t f_1+\innerproduct{Uu}{e_1}f_2}\leq 1\}\\
&=\sup\{|st+|v|a|:a,t\in\ar,\  \norma{t f_1+a f_2}\leq 1\}\\
&=\norma{s f_1^*+|v|f_2^*}_{E^*} = \norma{(s,v)}_{\mySpaceH{E^*}}.
\end{aligned}
\]
The inequality on the third line follows from  the inequality $|\innerproduct{Uu}{e_1}|\leq|u|$, valid for every $u\in\ell_2$.  The monotonicity of the basis $\{f_1,f_2\}$, as described in Lemma~\ref{easyLemma}, implies that  the supremum in the third line is taken on a larger set of elements $(t,u)$.
\medskip

So far we have the inequality $\norma{\Lambda_{s,v}}\leq  \norma{(s,v)}_{\mySpaceH{E^*}}$. To conclude the proof of ~(\ref{isometryToProve}), we show that for our already fixed $(s,v)$ there exists a norm-$1$ vector $(t,u)$ in $\mySpaceH{E}$, for which 
$\Lambda_{s,v}(t,u)=\norma{(s,v)}_{\mySpaceH{E^*}}$.

\medskip

Choose an element $a_1f_1 + a_2f_2\in E$ whose norm is~$1$, such that the functional $sf_1^*+|v|f_2^*$ attains at it its norm. That is,
\begin{equation}\label{linearFunctionalAttainsNorm}
sa_1+a_2|v|_2=\norma{sf_1^*+|v|f_2^*}_{E^*}=\norma{(s,v)}_{\mySpaceH{E^*}}. 
\end{equation}
If $v=0$, then $(a_1,0)$ has norm $1$ in $\mySpaceH{E}$, and by (\ref{linearFunctionalAttainsNorm}), $\Lambda_{s,0}(a_1,0)=\norma{(s,0)}_{\mySpaceH{E^*}}$.
If $v\neq 0$, then the element $(a_1,a_2v/|v|)$ has norm $1$ in $\mySpaceH{E}$, because by the sign-symmetry of the basis $\{f_1,f_2\}$,
\[
\norma{(a_1,a_2v/|v|)}_{\mySpaceH{E}}=\norma{a_1f_1+|a_2|f_2}=\norma{a_1f_1+a_2f_2}_E =1.
\]
Moreover, 
\[\Lambda_{s,v}(a_1,a_2v/|v|)=sa_1+a_2|v|=\norma{(s,v)}_{\mySpaceH{E^*}}.\]
This completes the proof of  (\ref{isometryToProve}).  
\end{proof}


\section{When is $\mySpaceH{E}$ an $L_1$-subspace?}
\subsection*{Preliminaries}
In (\cite{Dor}, proposition $1.3$),  Dor proved that if $E$ is a two-dimensional normed space and $\{e_1,e_2\}$ is a basis for $E$, then 
the functions ${\varphi(t)=\norma{e_1-te_2}}$ and ${\psi(t)=\norma{e_2-te_1}}$ determine an integral representation of the norm by the formula
\begin{equation}\label{DorRepresentationTheorem}
\norma{a_1e_1+a_2e_2}=\frac{1}{2}\int_{-\infty}^{\infty}|a_1x+a_2|\,d\varphi_{+}^{'}(x)+\left(\frac{\psi_{+}^{'}(0)-\psi_{-}^{'}(0)}{2}\right)|a_1|,
\end{equation}
for every $a_1,a_2\in\ar$. Since the functions $\varphi,\psi$ are both convex, the one-sided derivatives are bounded, non-decreasing
functions defined everywhere, so the integral in (\ref{DorRepresentationTheorem}) is a  Riemann-Stieltjes integral. Moreover,
the norm is smooth at $e_2$ if and only if $\psi$ is differentiable at zero, in which case only the integral remains on the right-hand side of 
(\ref{DorRepresentationTheorem}).

\medskip
The definition of the norm (\ref{normDefinition}) can be generalized as follows. For an arbitrary normed space $F$, 
we denote by $\ar\oplus_{\littleE} F$ the vector space $\ar\oplus F$ equipped with the norm 
$\norma{(t,x)}=\norma{(\norma{x}_{\scriptscriptstyle{F}})e_1+te_2}_{\littleE}$. Here $\{e_1,e_2\}$ is assumed to be a sign-symmetric basis for~$E$.
The main result of \cite{Dor} states that under certain smoothness assumptions on the norm $\norma{\cdot}_{\littleE}$, a necessary condition for the space $\ar\oplus_{\littleE} F$ to be an $L_1$-subspace is that $F$ is isometrically isomorphic to the Hilbert space of its dimension. (\cite{Dor}, Theorem~$1.5$). For $p>2$, the space $\ell_p^2$ satisfies the smoothness assumptions in the main result, and so by taking $E=F=\ell_p^2$,  Dor was able to deduce that $\ell_p^3$ is not an $L_1$-subspace if $p>2$, thus confirming a conjecture made by Bolker.

\medskip

 In view of Dor's main result, one may ask: \emph{what happens if $F$ is already a Hilbert space?} 
This question leads  to an investigation of the finite dimensional spaces $\directSum{\ell_2^n}$ and the infinite dimensional ones $\directSum{\ell_2}$. The corresoponding rotationally-symmetric norms are those given by (\ref{normDefinition}) in the introduction. For example, it is not immediately clear for which $p$, if at all,  the spaces $\ar\oplus_{\ell_p^2}\ell_2^n$ are $L_1$-subspaces. 
\begin{tcolorbox}
\begin{proposition} If $n>1$, then $\ar\oplus_{\ell_p^2}\ell_2^n$ is an $L_1$-subspace if and only if $1\leq p\leq 2$.
\end{proposition}
\end{tcolorbox}
\begin{proof}
If $p>2$, then Theorem~$3$ of \cite{KoldLonke} implies immediately that $\ar\oplus_{\ell_p^2}\ell_2^n$  is not an $L_1$-subspace. Assume now that $1\leq p\leq 2$. Let $\text{O}(n+1)$ denote the orthogonal group in $\Ren{n+1}$ and let
$\mu$ be its normalized Haar-measure. Let $G_{n+1}\subset \text{O}(n+1)$ denote the stabilizer of the vector $e_{n+1}$, with its normalized Haar-measure~$\nu$. Write each point $x\in\Ren{n+1}$ as $(y,t)$ where
$y\in\Ren{n}$ and $t\in\ar$. Then
\[
\begin{aligned}
\int_{G_{n+1}}\norma{gx}_p^p\,d\nu(g)&=\int_{\text{O}(n)}\norma{hy}_p^p\,d\mu(h)+|t|^p\\
&=\norma{y}_2^p\int_{\sphere{n-1}}\norma{u}_p^p\,d\sigma_{n-1}(u)+|t|^p\\
&=c_{n,p}\norma{y}_2^p+|t|^p,
\end{aligned}
\]
where $\sigma_{n-1}$ is the normalized rotation invariant measure on $\sphere{n-1}$. For each $g\in G_{n+1}$, the norm $x\to \norma{gx}_p$ is isometric to the norm $\norma{\cdot}_p$. Since
$\ell_p^{n+1}$ is isometrically isomorphic to a subspace of  $L_p$, the function $\exp(-\norma{gx}_p^p)$ is positive-definite in $\Ren{n+1}$. Products and pointwise limits of continuous positive definite functions are also 
positive definite, hence the function 
\[ x\to \exp(-\int_{G_{n+1}}\norma{gx}_p^p\,d\nu(g)),\qquad (x\in\Ren{n+1})\]
is also a positive definite function. Therefore $\Ren{n+1}$ with the norm $(c_{n,p}\norma{y}_2^p+|t|^p)^{1/p}$ is isometrically isomorphic to a subspace of $L_p$. By applying a linear
transformation we dedcue that $\ar\oplus_{\ell_p^2}\ell_2^n$  also isometric to a subspace of $L_p$, and when $1\leq p\leq 2$, this implies that it is also an $L_1$-subspace.
\end{proof}

\medskip

 \medskip
 \subsection*{Two preliminary lemmas}
 A random variable with values in $\Ren{n}$  
is called a \emph{random vector in $\Ren{n}$}. 
An \emph{isotropic vector in} $\Ren{n}$ is a random vector $X$ in $\Ren{n}$ whose probability distribution is invariant under orthogonal transformations. That is, 
for every orthogonal transformation $U$ in $\Ren{n}$, the probability-distributions of $U(X)$ and~$X$ are identical. 
\medskip

If $X$ is an isotropic  vector in $\Ren{n}$ then the probability-distributions of each of its components relative to some fixed orthonormal basis are all equal, and do not depend on the choice of the basis.
Hence the \emph{marginal distribution} of an isotropic vector can be defined as the probability-distribution of its first component, relative to the standard basis. \bigskip

\begin{tcolorbox}
\begin{lemma}\label{myTheorem} Let $\mu$ be  a symmetric probability measure on $\ar$, not concentrated at zero,
 with a finite first moment. If $E=\text{span}\{x,1\}$, then  $\mySpace{n+1}{E,\{x,1\}}$  is an $L_1$-subspace if and only
 if $\mu$ is a marginal distribution of an isotropic vector in $\Ren{n}$.
\end{lemma}
\end{tcolorbox}
\begin{proof} 
The identity function of $\ar$ belongs to $L_1(\ar,\mu)$ because $\mu$ has a finite first moment. The set $\{x,1\}$ is linearly independent, because in $L_1(\ar,\mu)$ the function~$x$ is proportional to some constant~$c$ if and only if 
$\mu\{x\neq c\}=0$, which by symmetry of $\mu$ implies $c=0$ and $\mu$ must then be concentrated at zero.
The functions $\{x,1\}$ form a sign-symmetric basis for their span in $L_1(\ar,\mu)$ because $\mu$ is symmetric.

 Assume  $J:\mySpace{n+1}{E}\hookrightarrow L_1$ is an isometric linear embedding.  The underlying vector space of
  $\mySpace{n+1}{E}$ is $\ar\oplus\ell_2^n$, with the standard basis $\{e_k\}_{k=1}^{n+1}$, where according to (\ref{normDefinition}) 
 \[ \norma{(t,a_1,\dots, a_n)}_{\mySpace{n+1}{E}}=\int_{\ar}|\bigg(\sum_{k=1}^na_k^2\bigg)^{1/2}x+t|\,d\mu(x).\]
 We proceed under the assumption that $J(e_1)= 1$. A justification for this assumption is deferred to avoid disrupting the main argument.
 
\medskip

Put $X_{k-1}=Je_k$ for $k=2,\dots n+1$.  If $(a_1,\dots,a_n) \in\Ren{n}$ is a unit vector, then since $J(e_1)=1$, 
\[ \int_0^1|\sum_{k=1}^na_kX_k(\omega)-t|\,d\omega=\int_{\ar}|\bigg(\sum_{k=1}^na_k^2\bigg)^{1/2}x-t|\,d\mu=\int_{\ar}|x-t|\,d\mu, \]
 for every $t\in\ar$.
Hence the probability distribution of $\sum_{k=1}^na_kX_k$ coincides with the probability distribution of $\mu$. This holds for every unit vector $(a_1,\dots,a_n)$, so that
the random vector $(X_1,\dots X_n)$ is isotropic, and its marginal distribution is $\mu$.

\medskip

Conversely, if $\mu$ is a marginal distribution of an isotropic vector $X$ in $\Ren{n}$, then  the map 
\[ J:\mySpace{n+1}{E}\hookrightarrow L_1,\]
defined by
\[ J(t,a_1,\dots, a_n)=\sum_{k=1}^na_kX_k+t , \]
is an isometric linear embedding, because
\[
\begin{aligned}
\norma{J(t,a_1,\dots,a_n)}_{L_1}&=\norma{\sum_{k=1}^na_kX_k+t}_{L_1}=\int_{\ar}|\bigg(\sum_{k=1}^na_k^2\bigg)^{1/2}x+t|\,d\mu\\
&=\norma{(t,a_1,\dots,a_n)}_{\mySpace{n+1}{E}}.\\
\end{aligned}
\] 

\medskip

It remains to justify the assumption that  $J(e_1)=1$, a justification that turns out to be much more involved than the main argument.
That $J(e_1)$ can assumed to be the constant function~$1$ was mentioned as a seemingly innocent normalization in Dor's paper (\cite{Dor}, p. 264) with no justification. However, not every subspace of $L_1$ is unital, that is, contains constant functions, and so some kind of justification is necessary.  It follows from Lemma 1.2 of \cite{DJP} that every nonzero subspace of $L_1$ is isometrically isomorphic to a unital subspace of $L_1$. In particular, if $F$ denotes the image of $\mySpace{n+1}{E}$ under $J$, then 
there is an isometry from $L_1$ onto itself that maps $F$ onto a unital subspace. 
The proof of Lemma 1.2 of \cite{DJP} reveals that in order to make sure that $J(e_1)$ is transformed to the constant function $1$ under such an isometry,
it suffices to check that $J(e_1)$ has full support in $F$, which means that if $f_1=J(e_1)$, then for every $g\in F$, 
\begin{equation}\label{fullSupport}
\lambda(\text{supp}(g)\backslash\text{supp}(f_1))=0,
\end{equation}
where $\lambda$ is the Lebesgue measure on $(0,1)$. Once we know that $f_1$ has full support in $F$, we can argue as in the proof of Lemma 1.2 of \cite{DJP} to produce an isometry $S:F\to L_1$ such that $S(f_1)= 1$, and then replace~$J$ by $S\circ J$. 

\medskip

To prove (\ref{fullSupport}), pick  $g\in F$, and scalars $\{a_i\}_{i=1}^{n+1}$, not all zero, such that 
\[ g=a_1f_1+J(\sum_{i=2}^{n+1}a_ie_i). \]
If $a_i=0$ for all $2\leq i\leq n+1$ then  (\ref{fullSupport}) is valid because $a_1\neq 0$. Otherwise, the vector
$v=\sum_{i=2}^{n+1}a_ie_i$ is nonzero,  so if $u=v/|v|$, then $g$ belongs to the span of $\{f_1,J(u)\}$.
By the definition of the norm in $\mySpace{n+1}{E}$, 
there is an isometric isomorphism
\begin{equation}\label{twoDimIsometry}
T:E\to\text{span}\{e_1,u\},\qquad T(s\cdot x+t\cdot 1)=su+te_1.
\end{equation}
If $J_{u}$ denotes the restriction of $J$ to the subspace $\text{span}\{e_1,u\}$, then $J_{u}\circ T$ is an isometry from $E$ onto $\text{span}\{f_1,J(u)\}$, that maps the constant function $1$ to the function $f_1$. 
By Lemma~3.4 of \cite{Hardin}, the function $f_1$ has full support in 
$\text{span}\{f_1,J(u)\}$. Since $g$ belongs to this two dimensional space, (\ref{fullSupport}) holds for every $g\in F$.
\end{proof}

The second Lemma is used in the proof of the main result of this section.
\begin{tcolorbox}
\begin{lemma} \label{subspaceLemma}
$\mySpaceH{E,\beta}$ is an {$L_1$-subspace} if and only if $\mySpace{n+1}{E,\beta}$ is
an {$L_1$-subspace} for every $n\geq 1$.
\end{lemma}
\end{tcolorbox}
\begin{proof}
The 'only if' part is clear. 
 
\medskip

 Assume that for every $n\geq1$ the space  $\mySpace{n+1}{E}$ is an $L_1$-subspace. In \cite{BDCK}, (Th\'eor\`eme~$2$, p.~$238$), it was proved that if $1\leq p\leq 2$, then a Banach space $X$ is isometrically isomorphic to a subspace of $L_p$ if and only if the function $\exp(-\norma{x}^p)$ is positive definite on~$X$.

To show that $\exp(-\norma{x})$ is a positive-definite function on $\mySpaceH{E}$, we must prove that if $\{x_1,\dots x_m\}$ is a finite
 subset of vectors in $\mySpaceH{E}$, and $(\xi_i)_{i=1}^m$ are  real scalars, then 
 \begin{equation}\label{wannaProveThis}
\sum_{1\leq k,l\leq m}\exp(-\norma{x_k-x_l})\xi_k\xi_l\geq 0.
\end{equation}

It is easy to verify that the space $\mySpaceH{E}$ is  the closed linear span of the nonzero vectors $\{e_i\}_{i=0}^{\infty}$. Moreover,
 for every choice of scalars $\{\alpha_i\}_{i=0}^{\infty}$ and positive integers $m<n$, the definition of the norm in
 $\mySpaceH{E}$ and Lemma~\ref{easyLemma} together imply that
\[ \norma{\sum_{i=0}^m\alpha_ie_i}_{\scriptscriptstyle\mySpaceH{E}}\leq \norma{\sum_{i=0}^n\alpha_ie_i}_{\scriptscriptstyle\mySpaceH{E}}.\]
It follows from  \cite{LT}, proposition 1.a.3,  that $\{e_i\}_{i=0}^{\infty}$ is a Schauder basis in $\mySpaceH{E}$.

\medskip

Let $P_n$ denote the natural projection in $\mySpaceH{E}$ onto the $n$ first basis vectors.
Since the range of $P_{n+1}$ is $\mySpace{n+1}{E}$, the assumption and the theorem referred to at the beginning of the proof entail that for every $n$,
\begin{equation}\label{haveThis}
\sum_{1\leq k,l\leq m}\exp(-\norma{P_nx_k-P_nx_l})\xi_k\xi_l\geq 0.
\end{equation}
Since $P_nx_k\to x_k$ as $n\to\infty$ for each $1\leq k\leq m$, (\ref{wannaProveThis}) follows from (\ref{haveThis}).
\end{proof}
 
 \medskip
 

We are ready for the main result of this section.
Recall that a real valued function $f(t)$ is said to be \emph{completely monotone in $(0,\infty)$}  if
$f^{(k)}(t)\geq 0$ for $t>0$ and $k=0,1,2,\dots$. If in addition $f$ is also right-continuous at zero, then
it is said to be \emph{completely monotone in $[0,\infty)$}.

\begin{tcolorbox}
\begin{theorem}\label{theCorollary}
Let $E=(\Ren{2},\|\cdot\|)$ such that $\{e_1,e_2\}$ is a sign-symmetric normalized basis, and the norm is smooth at $e_2$.
Put $\varphi(t)=\norma{e_1+te_2}$, and assume that $\varphi(t)$ is of class $C^2(\mathbb{R})$.
Then  $\mySpaceH{E}$ is an $L_1$-subspace if and only if the function $t\to\varphi''(\sqrt{t})$ is completely monotone for  $t\geq 0$.
\end{theorem}
\end{tcolorbox}
\begin{proof}
By (\ref{DorRepresentationTheorem})
\begin{equation}\label{DorRepresentation}
 \norma{a_1e_1+a_2e_2}=\frac{1}{2}\int_\ar |a_1x-a_2|\,\varphi''(x)\,dx, \qquad (a_1,a_2\in\ar). 
 \end{equation}

If $\mu$ is the measure on $\ar$ whose density with respect to the Lebesgue meausre is $\frac{1}{2}\varphi''$, then $\mu$ is a symmetric probability measure on $\ar$, and
(\ref{DorRepresentation}) shows that $E$ is isometrically isomorphic to the span of the functions $\{x,1\}$ in $L_1(\ar,\mu)$.  

\medskip

Assume now that $\mySpaceH{E}$ is an $L_1$-subspace. Fix $n\in\mathbb{N}$. The space  $\mySpace{n+1}{E}$ is an $L_1$-subspace, so by Lemma~\ref{myTheorem}, $\mu$ is a marginal distribution of an isotropic vector $X$ in $\Ren{n}$. If $\Phi_X$ is the characteristic function of $X$ and $\Phi_\mu$ is the characteristic function of~$\mu$, then orthogonal invariance implies:
\[ \Phi_X(Y)=E(e^{i Y\cdot X})=E(e^{i|Y|X_1})=\Phi_\mu(|Y|),\quad (Y\in\Ren{n}).\]
By a result of Schoenberg, (\cite{Schoenberg} , Theorem~$1$),   there exists a non-decreasing and bounded function $\alpha_n(u)$ defined on $[0,\infty)$ such that
\begin{equation}\label{finiteDimCFunctionRep}
\Phi_\mu(t)=\int_0^\infty \Omega_n(tu)\,d\alpha_n(u),\qquad (t\in\ar),
\end{equation}
where $\Omega_n$ is the characteristic function of the marginal distribution of a unit vector uniformly distributed over the sphere $\sphere{n-1}$. 

These considerations hold for all $n$. By another result of Schoenberg (\cite{Schoenberg}, Theorem~$2$),  there exists a non-decreasing and bounded function $\sigma(u)$ defined on $[0,\infty)$ such that
\begin{equation}\label{infiniteDimCFunctionRep}
\Phi_\mu(t)=\int_0^\infty e^{-t^2u}\,d\sigma(u), \qquad (t\in\ar).
\end{equation}

The characteristic function of $\mu$ is the Fourier transform of its  density. Hence
\begin{equation}\label{fundEq}
(\varphi'')^{\wedge}(t)=2\int_0^{\infty}e^{-t^2u}\,d\sigma(u), \qquad (t\in\ar).
\end{equation}
As a function of $t$, the right hand side of (\ref{fundEq}) is in $L_1(\ar)$. Take the Fourier transforms of both sides and
use the fact that $\varphi''$ is continuous and even to deduce that 
\[ \varphi''(t)=2\int_0^{\infty}e^{-t^2/4x}\sqrt{\frac{\pi}{x}}\,d\sigma(x), \qquad (t\in\ar).\]
It follows that $\varphi''(\sqrt{t})$ is completely monotone in $[0,\infty)$. (See also \cite{BN}, p~$223$).

\medskip

These implications can be reversed. If $\varphi''(\sqrt{t})$ is completely monotone in $[0,\infty)$, then $(\varphi'')^{\wedge}(\sqrt{t})$ is also completely monotone in $[0,\infty)$.
By Bernstein's theorem, (\cite{Bernstein1928}, p.$56$), there exists a non-decreasing bounded function $\sigma$ such that
(\ref{infiniteDimCFunctionRep}) and (\ref{fundEq}) are valid. For every $n\in\mathbb{N}$, Schoneberg's theorem can now be
applied in the reverse direction, to deduce  from (\ref{infiniteDimCFunctionRep}) a representation of the form (\ref{finiteDimCFunctionRep}). As a result,  for every $n$ the measure $\mu$ is a marginal distribution of an isotropic vector in $\Ren{n}$.  By Lemma~\ref{myTheorem}, the space $\mySpace{n+1}{E}$ is an $L_1$-subspace for all $n$, and by Lemma~\ref{subspaceLemma}, so is $\mySpaceH{E}$.
\end{proof}

\begin{remark}
To obtain the integral representation (\ref{DorRepresentation}) it is not necessary to assume that 
the second derivative $\varphi''$ is continuous everywhere. By a result of Koldobsky, the second derivative $\varphi''$,
viewed as a distribution in the sense of Schwartz, is a positive, finite measure with a finite first moment. Moreover,
if we assume that $\varphi''$ has no atom at zero, and is continuous for $t\neq 0$, and also assume that the norm is smooth
at $e_2$, then the integral representation (\ref{DorRepresentation}) still holds. Under these assumptions,
the arguments of the first part of the proof of the preceding theorem can be applied, but when $\varphi''(t)$ is recovered
back from its Fourier transform, there is no information about $t=0$, and so we can only deduce that $\varphi''(\sqrt{t})$ is
completely monotone in $(0,\infty)$. This can be applied, for instance, to show that the function
\[ f_p(t)=t^{\frac{p}{2}-1}(t^{p/2}+1)^{\frac{1}{p}-2}\qquad (t>0), \]
is completely monotone in $(0,\infty)$ if $1\leq p\leq 2$. (If $p>2$ it is not even monotone.) More details are to be found in the appendix.

\end{remark}

\section{A finite dimensional Inversion Formula}

Theorem~\ref{theCorollary} suggests that having a single inequality $(-1)^m\frac{d^m}{dt^m}(\varphi''(\sqrt{t}))\geq 0$ for some $m$ has something to do with $\mySpace{k+1}{E}$ being an $L_1$ subspace for some $k$ that depends on $m$. Under an additional smoothness assumption, this is indeed the case.

\medskip

\subsection{The cosine transform}
We begin by recalling some basic concepts and results from  convex and differential geometry.

\medskip

 In the previous section we
discussed probability distributions. In this section we are going to discuss the concept of distributions in the sense of Schwartz. In \cite{Boas}, Boas remarked that
"It has not helped communication that 'distribution' now means different things in probability and in functional analysis". The reader
is asked to interpret every occurence of the word 'distribution' that appears in this section, as a distribution in the sense of Schwartz.
\medskip

 If $\testFunctionSpace$ denotes the space of infinitely differentiable and even functions on the
sphere, then for $f\in\testFunctionSpace$ the \emph{cosine transform} $\CT f$ of $f$ is the function
\begin{equation}\label{cosineTransformDefinition}
({\mathcal T}f)(u)=\int_{\sphere{n-1}}|\innerproduct{u}{v}|f(v)\,d\lambda_{n-1}(v), \quad (u\in\sphere{n-1}),
\end{equation}
where $\lambda_{n-1}$ denotes the $(n-1)$-dimensional spherical Lebesgue measure.

\medskip

The cosine transform is continuous bijection of $C_e^{\infty}(\sphere{n-1})$ onto itself, 
and can be extended by duality to a bi-continuous bijection of the dual space $D_e(\sphere{n-1})$ of even distributions on $\sphere{n-1}$.
Thus, if $\rho\in D_e(\sphere{n-1})$ , then $\CT\rho$ and $\inverseCT\rho$ are even distributions defined via duality by:
\[ (\CT\rho)(f)=\rho(\CT f),\ \ \mbox{and}\ \ (\inverseCT\rho)(f)=\rho(\inverseCT f),\quad (f\in\testFunctionSpace). \]

\noindent For further details, see \cite{GardnerBook}, \cite{GW1992} and \cite{weil76}.

\medskip

Note that if in (\ref{cosineTransformDefinition}) the function $f$ is nonnegative, even and integrable, then the integral is just the norm of the function $f_v(u)=\innerproduct{u}{v}$ in the space $L_1(\sphere{n-1},fd\lambda_{n-1})$. In general, a normed space $(\Ren{n},\norma{\cdot})$ is
an $L_1$-subspace if and only if there exists a positive, even and finite measure $\mu$ on the sphere $\sphere{n-1}$, such that $(\CT\mu)(x)=\norma{x}$ for every $x\in\Ren{n}$.
In this case, the representation of the norm as
\begin{equation}\label{LevyRepresentation}
\norma{x}=\int_{\sphere{n-1}}|\innerproduct{x}{v}|\,d\mu(v)\qquad (x\in\Ren{n})
\end{equation}
is called the \emph{L{\'e}vy representation} of the norm. (See Lemma 6.4 of \cite{kol1992}).

The restriction of a norm in $\Ren{n}$ to the sphere belongs to $\distributionSpace$. Since the cosine transform is a bijection, there exists a unique distribution $\rho\in\distributionSpace$ such that $\norma{\cdot}=\CT\rho$. If $\rho$ is a positive distribution, then it is a measure, and we have  a L{\'e}vy representation. Hence, to determine whether $(\Ren{n},\norma{\cdot})$ is an $L_1$-subspace,  we need to determine whether the inverse cosine transform of the restriction of the norm to the sphere is a positive distribution.

To this end, a useful tool is an inversion formula, due to Goodey and Weil.
\begin{equation}\label{GWinversionFormula}
\inverseCT=\frac{1}{2\omega_{n-2}}(\Delta_n+n-1){\mathcal R}^{-1},
\end{equation}
where $\Delta_n$ denotes the Laplace-Beltrami operator on $\sphere{n-1}$, $\omega_{n-2}=\frac{2\pi^{(n-1)/2}}{\Gamma((n-1)/2)}$ and
${\mathcal R}^{-1}$ is the inverse spherical Radon transform. (\cite{GW1992}, Proposition $2.1$). In the rotationally-symmetric case, this inversion formula can be effectively applied to functions of one variable. 

\medskip

\subsection*{Rotationally symmetric distributions}

We say that a  distribution $\rho\in D_e(\sphere{n-1})$ is rotationally symmetric if for every $g\in C_e^\infty(\sphere{n-1})$ and every rotation $V$ that keeps the rotation axis  fixed,
$\rho(g\circ V)=\rho(g)$. It is not hard to verify that if $\rho$ is a rotationally symmetric distribution, then $\CT \rho$ and $\inverseCT \rho$ are also rotationally symmetric.

\medskip
For rotationally symmetric functions, and even dimensions, the inversion formula 
 (\ref{GWinversionFormula}) is significantly more tractable than for \textcolor{brown}{in?} the general case. If $f$ is a rotationally symmetric  continuous function on the sphere, then
for each $0\leq\theta\leq\pi$, it is constant on the sets $\{u\in\sphere{n-1}: u_n=\cos\theta\}$, where $u_n$ is the $n$'th coordinate of~$u$.
 We use the same symbol $f$ to denote the induced function of $\theta$. Therefore, if we write $f(\theta)$ we are referring
 to the value of $f$ at the point $\sin\theta e_{n-1}+\cos\theta e_n\in\sphere{n-1}$. Since the inverse cosine transform of $f$ is also rotationally symmetric,  it ought to be described in terms of the one-variable function $f(\theta)$.   This is indeed the case.
 
 \medskip
 
In  \cite{LonkeDegreeOf} and in the references therein, it is explained that if $n\geq 4$ is even, then the inverse cosine transform of $f$  can be calculated via the inversion formula (\ref{GWinversionFormula}) by appyling the following differential operator to the function $x^{n-3}f(\sin^{-1}x)$
\begin{equation}\label{differentialOperator}
c_n\left [(1-x^2)\frac{d^2}{dx^2}-(n-1)x\frac{d}{dx}+n-1\right]\circ x\left (\frac{1}{x}\frac{d}{dx}\right)^{\frac{n}{2}-1}.
\end{equation}
Therefore the inverse cosine transform of $f$ induces a one-dimensional  distribution on the interval $(0,1)$, which is defined by a differential operator
of degree ${\frac{n}{2}+1}$. As before, we shall use the symbol $\inverseCT f$ to also denote the one-dimensional distribution induced by the inverse cosine transform of the function~$f$. The following inversion formula recovers the inverse cosine transform
away from the pole and its orthogonal complement.
\subsection*{An inversion formula}

\bigskip

\begin{theorem}\label{myInversionTheorem}
Let $f\in C^{\infty}(\Ren{2n}\backslash\{0\})$ be a rotationally symmetric, continuous, and homogeneous of degree~$1$ real valued function. Put
\[ \varphi(t)=f(e_{2n-1}+te_{2n}),\qquad (t\in\ar). \]
Then the rotationally symmetric inverse cosine transform  $\inverseCT f$ restrticed to  the inverval $(0,1)$ is given by the formula 

\begin{equation}\label{myInversionFormula}
 \alpha_{2n}^{-1}(\inverseCT f)=\frac{(-2)^{n-1}}{x^{2n+1}}\frac{d^{n-1}}{dt^{n-1}}(\varphi''(\sqrt{t}))\Big|_{t\to (\frac{1-x^2}{x^2})}.
\end{equation}
\end{theorem}
\begin{proof} 

For every $k$-times differentiable function $f$ defined in $\ar^+$,
\begin{equation}\label{keyObservation}
\frac{1}{2^k}\left(\frac{1}{x}\frac{d}{dx}\right)^k(f)=\frac{d^k}{dt^k}(f{\smallcirc}\sqrt{t})\Big|_{t\to x^2},
\end{equation}
as can easily be verified by induction on $k$.

\medskip

Since $f$ is homogeneous of degree $1$,  we have 
\begin{equation}\label{homogeneity}
 x^{2n-3}f(\sin^{-1}x)=x^{2n-2}\varphi(\frac{\sqrt{1-x^2}}{x}),\qquad (0<x <1).
 \end{equation}
For $0<t<1$, put  $h(t)=(\frac{1-t}{t})^{1/2}$. In view of (\ref{keyObservation}) and (\ref{homogeneity}), we can write
\begin{equation}\label{twoDiffops}
\frac{1}{2^{n-1}}\inverseRadonPart{x}{n-1}(x^{2n-3}f(\sin^{-1} x)) = \frac{d^{n-1}}{dt^{n-1}}(t^{n-1}\varphi{\smallcirc} h(t))\Big|_{t\to x^2}.
\end{equation}
For $m\in\mathbb{N}$, define a differential operator on $C^{\infty}(0,\infty)$ by
\[ F_m(g)=(1-x^2)\frac{d^2g}{dx^2}-(m-1)x\frac{dg}{dx}+(m-1)g,\]
and for every $m\in\mathbb{N}$ and $g\in C^{\infty}(0,\infty)$, define a function of $0<t<1$ by
\[ H_m^g(t)=\frac{d^{n-1}}{dt^{n-1}}(t^{n-1}g{\smallcirc} h). \]
By (\ref{differentialOperator}) and (\ref{twoDiffops}),  for $0<x<1$ the formula (\ref{myInversionFormula})  is 
equivalent to 

\begin{equation}\label{eq1}F_{2n}(x H_n^{\varphi}(x^2)))=
\frac{(-1)^{n-1}}{x^{2n+1}}
\frac{d^{n-1}}{dt^{n-1}}
(\varphi''(\sqrt{t}))\Big|_{t\to h^2(x^2)}
\end{equation}

\medskip

An application of the differential operator $F_{2n}$ to the expression $xH_n^{\varphi}(x^2)$ is straightforward:
\begin{equation}\label{eq2}
F_{2n}(xH_n^\varphi(x^2))=(H_n^\varphi)'(x^2)(6x-4x^3(n+1))+4x^3(1-x^2)(H_n^\varphi)''(x^2).
\end{equation}

It remains to show that the right hand sides of (\ref{eq1}) and (\ref{eq2}) are equal. If~${n=1}$, a routine computation verifies the equality.

\medskip

For every $0<a<b<1$, the right hand sides of  (\ref{eq1}) and (\ref{eq2}) are  both linear in~$\varphi$ and continuous in 
$C^{\infty}[a,b]$. Since the set of polynomials is dense in $C^{\infty}[a,b]$ with the topology of uniform convergence of all derivatives, it suffices to check that (\ref{eq1}) and (\ref{eq2})  agree on all monomials $\varphi(x)=x^m$, where $m\in\mathbb{N}$. Consider two cases:
\begin{enumerate}
\item[1.] $m=2l$ is an even integer and $l<n$. Hence $\varphi''(\sqrt{x})=2l(2l-1)x^{l-1}$. Since $l-1<n-1$, the right hand side of (\ref{eq1}) vanishes. On the other hand,
\[ x^{n-1}\varphi(h(x))=x^{n-l-1}(1-x)^l \]
is a polynomial of degree $n-1$, so with $\varphi=x^{2l}$, the function $H_n^\varphi$ is  constant, and the right hand side of (\ref{eq2}) vanishes as well.
\item[2.] In the second case,  $m$ is either an even integer at least $2n$, or else an odd natural number. With $\varphi=x^m$, expand the right hand side of 
\[ H_n^\varphi(x)=\frac{d^{n-1}}{dx^{n-1}}(x^{n-1-m/2}(1-x)^{m/2}), \]
to obtain
\[ (-1)^{n-1}\sum_{k=0}^{n-1}{n-1\choose k}(-1)^k(\frac{1-x}{x})^{m/2-n+k+1}(n-1-m/2)^{(k)}(m/2)^{(n-k+1)},\]
where  $(a)^q=a(a-1)\dots (a-q+1)$ for real $a$ and $q\in\mathbb{N}$. 
A substitution of this sum for $H_n^\varphi(x)$ into the right hand side of (\ref{eq2}) results in a complicated but elementary calculation that yields the expression 
\begin{equation}\label{final}
\frac{(-1)^{n-1}}{x^{m+1}}m(m-1)(1-x^2)^{m/2-n}(m/2-1)^{(n-1)}, 
\end{equation}
which for  $\varphi(x)=x^m$, is precisely the right hand side of (\ref{eq1}). 
\end{enumerate}
This holds for every compact interval $[a,b]\subset (0,1)$ and therfore holds for every ${\varphi\in C^{\infty}(0,1)}$.
\end{proof}

\medskip

\subsubsection*{Comment}

 Here is an example that demonstrates how even very "gentle" smooth perturbations of the Euclidean two dimensional ball may yield norms for which the infinite dimensional space $\mySpace{}{E}$ is not an $L_1$ subspace. Compare this to the finite dimensional situation, where in \cite{Schneider} Schneider showed that smooth perturbations of the Euclidean ball, with respect to a certain metric that depends on a finite number of high-order derivatives of the norm, can produce zonoids whose polars are zonoids.  In the present context, complete monotonicity, as well as the inverse Laplace transform,  both depend on infinitely many high-order derivatives, which renders the property of $\mySpace{}{E}$ being an $L_1$-subspace unstable with respect to arbitrarily small smooth perturbations.

To construct such perturbations, fix two real numbers $\alpha$ and $\beta$, and let~$g(\theta)$ be a solution of the differential equation 
\begin{equation*}\label{eq:diffeqLaplace}
(g+g'')(\theta)=\alpha\cos^2\theta+\beta\sin^2\theta
\end{equation*}
  in $[0,2\pi]$, such that $g$ has period $\pi$. Choose $\eps>0$ sufficiently small, so that
 $1+\eps g(\theta)$ is a restriction of a norm to the unit circle. That is, there is some norm $\norma{\cdot}$ such that 
 $$\nu(\theta):=\norma{\cos\theta e_1+\sin\theta e_2}=1+\eps g(\theta).$$
Next, consider the function we have met before:
\[ \phi(t)=\norma{e_1+te_2}=\sqrt{1+t^2}(1+\eps g(\arctan t)). \]
Thus
\[ \phi''(t)=\frac{1+\eps(g+g'')(\arctan t)}{(1+t^2)^{3/2}}=\frac{1}{(1+t^2)^{3/2}}+\frac{\eps(\alpha+t^2\beta)}{(1+t^2)^{5/2}}.\]
Hence, the inverse Laplace transform of $\phi^{''}(\sqrt{t})$ is:
$$\frac{2e^{-t}\sqrt{t}}{3\sqrt{\pi}}(3(1+\beta\eps)+2t\eps (\alpha-\beta))$$
and this is positive for $\alpha\geq \beta\geq -1/\eps$, but eventually becomes negative if $\alpha<\beta$.

\section{The One Gaussian Norm}

The simplest positive measure that can occur on the r.h.s of (\ref{fundEq}) is a unit mass measure concentrated at some nonzero point, e.g., $e^{-t^2}dt$.  Therefore, 
consider the  function $\phi(t)$ satisfying the equation $\phi''(t)=2e^{-t^2}$, so that the two-dimensional corresponding norm is
\begin{tcolorbox}

\begin{equation}\label{oneGaussian}
\gamma_1(x,y)=\int_{-\infty}^{\infty}|xt-y|e^{-t^2}\,dt=|x|e^{-y^2/x^2}+|y|\Phi(|y/x|),
\end{equation}
where
\[ \Phi(t)=2\int_{0}^te^{-s^2}\,ds \]
\end{tcolorbox}
Henceforth the norm $\gamma_1$ will be referred to as the \emph{One Gaussian Norm}.
The proof of the next Lemma makes full use of the integral representation (\ref{oneGaussian}) of the One Gaussian Norm.
\begin{tcolorbox}
\begin{lemma}\label{strictlyConvexAndSmooth}
The unit ball $B_1$ of the One Gaussian Norm is sign-symmetric, smooth and strictly convex.
\end{lemma}
\end{tcolorbox}
\begin{proof}
That $B_1$ is sign-symmetric follows at once from~(\ref{oneGaussian}).  Smoothness will follow from a direct computation of the derivative of $\gamma_1$ in subsection 3.1 below. To prove that $B_1$ is strictly convex, i.e., that the boundary $\partial B_1$ contains no intervals, pick any two  points $(x_1,y_1),(x_2,y_2)$ in $\partial B_1$. 
If $x_1=0$, then $|y|=\pi^{-1/2}$.
Consider the functions $f_1(t)=x_1t-y_1$, $f_2(t)=x_2t-y_2$ defined in $\ar$. We have 
\[ 
\begin{aligned}
\gamma_1(\frac{(x_1,y_1)+(x_2,y_2)}{2})=1&\Longleftrightarrow \int_\ar |f_1(t)+f_2(t)|e^{-t^2}\,dt\\
&=\int_\ar |f_1(t)|e^{-t^2}\,dt+\int_\ar |f_2(t)|e^{-t^2}\,dt 
\end{aligned}
\]
Equality in the triangle inequality in $L_1(e^{-t^2}\,dt)$ for the continuous functions $f_1,f_2$, implies that they must have the same sign everywhere in $\ar$:
\[ (x_1t-y_1)(x_2t-y_2)\geq 0\qquad (\forall t\in\ar) \]
If $x_1=0$, then $|y_1|=\pi^{-1/2}$, which implies that $x_2t-y_2$ has a constant sign in $\ar$, which is possible only if $x_2=0$, whence $|y_2|=\pi^{-1/2}$. Consequently,
the two points are equal or antipodal. They can't be antipodal because their average was supposed to have norm $1$, so they must be equal. Thus we may assume that both $x_1,x_2$ are nonzero. Examining the  quadratic above we deduce that $x_1x_2\geq 0$, $y_1y_2\geq 0$, and  $\frac{y_1}{x_1}=\frac{y_2}{x_2}$, which implies that $(x_1,y_1)$ and $(x_2,y_2)$ are proportional. Since both have norm one, and they are not antipodal, they must be equal. This proves that there are no intervals in the boundary of $B_1$.
\end{proof}
\subsection{Calculation of the One-Gaussian support function}
As before, 
\[ \Phi(t)=2\int_{0}^te^{-s^2}\,ds\]
\begin{tcolorbox}
\begin{lemma}\label{oneGaussianSupportFunction}
With $z\in\ar$,let  $G(z)=e^{z^2}\Phi(z)$, and $G^{-1}$ its inverse function. The support function of the One-Gaussian is given by
\begin{equation}\label{eq:oneGaussianSupportFunction1}
h(\cos\theta,\sin\theta)=\cos\theta\exp(G^{-1}(\tan\theta)^2),\qquad (-\frac{\pi}{2}<\theta<\frac{\pi}{2})
\end{equation}
\[ h(0,1)=\lim_{\theta\uparrow\frac{\pi}{2}}h(\cos\theta,\sin\theta)=\frac{1}{\sqrt{\pi}}\]
and
\[ h(0,-1)=\lim_{\theta\downarrow -\frac{\pi}{2}}h(\cos\theta,\sin\theta)=\frac{1}{\sqrt{\pi}}\]
\end{lemma}
\end{tcolorbox}
\begin{proof}
Differentiating (\ref{oneGaussian}) yields 
\begin{equation}\label{oneGaussianNormGradient}
\nabla\gamma_1(x,y)=
\left\{
\begin{aligned}
&(\sgn(x)e^{-y^2/x^2},\sgn(y)\Phi(|y|/|x|))& x\neq 0\\
&\sqrt{\pi}(0,\sgn(y)) & x=0
\end{aligned}
\right.
\end{equation}
and
\begin{equation}\label{gradientAtEndPoints}
\nabla \gamma_1(0,\pm 1)=\lim_{x\to 0}\nabla\gamma_1(x,\pm 1)=(0,\pm \sqrt{\pi})
\end{equation}
Thus $\gamma_1$ is continuously differentiable in $\Ren{2}$ , which provides a  proof to the fact that $B_1$ is smooth. Being also strictly convex, the \emph{reverse spherical map} is well defined. This is the map that sends each unit vector $u_\theta=(\cos\theta,\sin\theta)$ to the unique boundary point ${\rho(\theta)\in \partial B_1}$ at which $u_\theta$ is an outer normal. By (\cite{SchneiderBook}, (1.39) p. 53), we have
\begin{equation}\label{SchneidersFormula}
\nabla\gamma_1(\rho(\theta))=\frac{u_\theta }{h(u_\theta)}
\end{equation}

The derivative $\nabla\gamma_1$ is homogeneous of degree zero, so it suffices to examine it on unit vectors.
Given a unit vector $u_\alpha=(\cos\alpha,\sin\alpha)$, we have by (\ref{oneGaussianNormGradient})
\[\nabla\gamma_1(u_\alpha)=(\sgn(\cos\alpha)e^{-\tan^2\alpha},\sgn(\sin\alpha)\Phi(|\tan\alpha|)),\qquad (\frac{\pi}{2}<\alpha<\frac{\pi}{2}) \]
and for $\alpha=\pm\frac{\pi}{2}$ the values are given by (\ref{gradientAtEndPoints}). 
Given $\theta$,  (\ref{SchneidersFormula}) implies that the vector $\nabla\gamma_1(u_\alpha)$ is proportional to $u_\theta$, if and only if 
the vector $u_\alpha$ is proportional to $\rho(\theta)$.
The vector $\nabla\gamma_1(u_\alpha)$ is proportional to $u_\theta$ if and only if 
\[\tan\theta=\sgn({\tan\alpha})\Phi(|\tan\alpha)|)e^{\tan^2\alpha}=\Phi(\tan\alpha)e^{\tan^2\alpha},\ \ (-\frac{\pi}{2}< \alpha <\frac{\pi}{2})\]
(we used here the fact that $\Phi(|x|)=\sgn(x)\Phi(x)$ for every $x\in\mathbb{R}$). 
 Then, $\tan\theta=G(\tan\alpha)$, so that $\tan\alpha=G^{-1}(\tan\theta)$, and 
\[ \cos\alpha=\frac{1}{\sqrt{1+(G^{-1}(\tan\theta))^2}},\quad\sin\alpha=\frac{G^{-1}(\tan\theta)}{\sqrt{1+(G^{-1}(\tan\theta))^2}}\]
Therefore, the reverse-spherical map is given by:
\[
\begin{aligned}
\rho(\theta)=\frac{(\cos\alpha,\sin\alpha)}{\gamma_1(\cos\alpha,\sin\alpha)}&=\frac{(1,G^{-1}(\tan\theta))}{\gamma_1(1,G^{-1}(\tan\theta))}\\
&=\frac{(1,G^{-1}(\tan\theta))}{\exp(-G^{-1}(\tan\theta)^2)+|G^{-1}(\tan\theta)|\Phi(|G^{-1}(\tan\theta)|)}\\
&=\frac{(1,G^{-1}(\tan\theta))}{\exp(-G^{-1}(\tan\theta)^2)+G^{-1}(\tan\theta)\Phi(G^{-1}(\tan\theta))}
\end{aligned}
\]
Here we used the fact that $\Phi(z)$ is an odd function. Now, once we have the reverse spherical map, we can compute the support function:
\[ h(u_\theta)=\innerproduct{u_\theta}{\rho(\theta)},\quad (u_\theta\in\mathbb{S}^1)
\]
This gives
\[
\begin{aligned}
h(\cos\theta,\sin\theta)&=\frac{\cos\theta+\sin\theta\cdot G^{-1}(\tan\theta)}{\exp(-G^{-1}(\tan\theta)^2)+G^{-1}(\tan\theta)\Phi(G^{-1}(\tan\theta))},\ \ (-\frac{\pi}{2} < \theta < \frac{\pi}{2})\\
\\
&=\cos\theta\exp(G^{-1}(\tan\theta)^2)
\end{aligned}
\]
The last equality follows from the identity
\[ \frac{1+xG^{-1}(x)}{\exp(-G^{-1}(x)^2)+G^{-1}(x)\Phi(G^{-1}(x))}=\exp(G^{-1}(x)^2)\quad (x\in\ar) \]
Since $\lim_{\theta\uparrow\pi/2}G^{-1}(\tan\theta)=\infty$, and  $\lim_{\theta\downarrow -\pi/2}G^{-1}(\tan\theta)=-\infty$,we have
\[ h(0,1)=\lim_{\theta\uparrow\frac{\pi}{2}}h(\cos\theta,\sin\theta)=\frac{1}{\sqrt{\pi}}\]
and
\[ h(0,-1)=\lim_{\theta\downarrow -\frac{\pi}{2}}h(\cos\theta,\sin\theta)=\frac{1}{\sqrt{\pi}}\]
\end{proof}
With $E=(\Ren{2},\gamma_1)$, the function that determines whether $\mySpaceH{E^*}$ is an $L_1$ subspace has a surprisingly compact form, albeit not easily analysed.
Given $t\geq 0$, put $\theta=\arctan t$. Then $(1,t)=\sqrt{1+t^2}(\cos\theta,\sin\theta)$, hence
\begin{equation}\label{supportFunctionIdentity}
h(1,t)=\frac{1+tG^{-1}(t)}{\exp(-G^{-1}(t)^2)+G^{-1}(t)\Phi(G^{-1}(t))}=\exp(G^{-1}(t)^2)
\end{equation}
Where $G(t)$ is as in Lemma (\ref{oneGaussianSupportFunction}).  Hence, with $\phi(t)=h(1,t)$,
\begin{equation}\label{secondDerivativeSupport}
\phi''(t)=\frac{\exp(G^{-1}(t)^2)}{2(1+tG^{-1}(t))^3}.
\end{equation}

The infinite dimensional space generated
by the One Gaussian Norm is isometric to the Banach space $\mathbb{G}_1\oplus \ar\subset L_1$, where $\mathbb{G}_1$ is a Gaussian Hilbert Space embedded in $L_1$. The results of the previous sections show that  $\phi''(\sqrt{t})$ is completely monotone if and only if $(\mathbb{G}_1\oplus \ar)^*$ is an $L_1$ subspace. During the preparation of this work, various numerical results led me to conjecture that the function $\phi^{''}(\sqrt{t})$ is indeed completely monotone. {private} A recent work
by {xx} proves this conjecture. 
This  immediately provides counterexamples to both Grothendieck's question and Schneider's problem, 
because  $\mathbb{G}_1\oplus\ar$ contains a sequence of finite dimensional sections that are both zonoids and polars of zonoids,  
the distances to $\ell_2^n$ of which are all equal to the distance of the One Gaussian Norm to the two-dimensional Euclidean ball.

\printbibliography

@misc{RyaZvav,
	author = {Ryabogin, D. and Zvavitch, A.},
	howpublished = {arXiv:2609.10852},
	month = {September},
	title = {Zonoids whose polars are zonoids: the Banach-Mazur distance need not tend to one},
	year = {2026}}

@misc{KoldobskyPrivate,
	author = {Koldobsky, A.},
	month = {July},
	title = {Private Communication},
	year = {2022}}

@misc{SchneiderPrivate,
	author = {R. Schneider},
	month = {January},
	title = {Private Communication},
	year = {2024}}

@article{KoldLonke,
	author = {Koldobsky, A. and Lonke, Y.},
	journal = {Bull. London Math. Soc.},
	month = {November},
	number = {6},
	pages = {693-699},
	title = {A short proof of Schoenberg's conjecture on positive definite functions},
	volume = {31},
	year = {1999}}

@article{Boas,
	author = {R.P. Boas},
	journal = {Amer. Math. Monthly},
	month = {December},
	number = {10},
	pages = {727--731},
	title = {Can We Make Mathematics Intelligible?},
	volume = {88},
	year = {1981}}

@misc{GrundKobos,
	author = {F. Grundbacher and T. Kobos},
	howpublished = {arXiv:2407.08829v2 [math.MG]},
	month = {June},
	title = {On certain extremal Banach-Mazur distances and Ader's characterization of distance ellipsoids},
	year = {(2025)}}

@article{Bernstein1928,
	author = {S. Bernstein},
	journal = {Acta Mathematica},
	pages = {1--66},
	title = {Sur les fonctions absolument monotones},
	volume = {51},
	year = {1928}}

@article{LonkeDegreeOf,
	author = {Y. Lonke},
	journal = {Arch. Math.},
	pages = {343-349},
	title = {On the degree of generating distributions of centrally symmetric convex bodies},
	volume = {69},
	year = {1997}}

@book{GardnerBook,
	author = {R. Gardner},
	publisher = {Cambridge University Press (Encyclopedia of mathematics and its applications; v. 58)},
	title = {Geometric tomography},
	year = {2006}}

@article{BN,
	author = {F. Barthe and A. Naor},
	journal = {Discrete Comput Geom},
	pages = {215--226},
	title = {Hyperplane Projections of the Unit Ball of $\ell_p^n$},
	year = {2002}}

@book{LT,
	author = {J. Lindenstrauss and L. Tzafriri},
	publisher = {Springer},
	title = {Classical Banach Spaces},
	volume = {I},
	year = {1977}}

@book{SchneiderBook,
	author = {R. Schneider},
	edition = {Second edition},
	publisher = {Cambridge},
	title = {Convex Bodies: The Brunn-Minkowski Theory},
	year = {2014}}

@article{GW1992,
	author = {Goodey, P. and Weil, W.},
	journal = {Journal of Differential Geometry},
	pages = {675--688},
	title = {Centrally symmetric convex bodies and the spherical Radon transform},
	volume = {35},
	year = {1992}}

@article{LindExtPaper,
	author = {Lindenstrauss, J.},
	journal = {Illinois J. Math.},
	pages = {488--499},
	title = {On the extension of operators with finite-dimensional range},
	volume = {8},
	year = {1964}}

@article{Herz,
	author = {Herz, C.S.},
	journal = {Proc. Amer. Math. Soc.},
	number = {4},
	pages = {670--676},
	title = {A class of negative-definite functions},
	volume = {14},
	year = {1963}}

@article{Ferg,
	author = {Ferguson, T.},
	journal = {The Annals of Mathematical Statistics},
	number = {4},
	pages = {1256--1266},
	title = {A representation of the symmetric bivariate Cauchy distribution},
	volume = {33},
	year = {1962}}

@article{AFJS,
	author = {Arias, A. and Figiel, T. and Johnson, W.B. and Schechtman, G.},
	journal = {Trans. Amer. Math. Soc.},
	pages = {3835--3857},
	title = {Banach spaces with the $2$-Summing Property},
	volume = {347},
	year = {1995}}

@article{BDCK,
	author = {Bretagnolle, J. and Dacunha-Castelle, D. and Krivine, J.L.},
	journal = {Ann. Inst. H. Poincar\'e Probab. Statist.},
	pages = {231--259},
	title = {Lois stable et espaces $L_p$},
	volume = {2},
	year = {1966}}

@article{Bolker,
	author = {E. Bolker},
	journal = {Trans. Amer. Math. Soc.},
	pages = {323--344},
	title = {A class of convex bodies},
	volume = {145, November},
	year = {1969}}

@article{DJP,
	author = {Delbaen, F. and Jarchow, H. and Pe{\l}czy{\'n}ski, A.},
	journal = {Positivity},
	pages = {339--367},
	title = {Subspaces of $L_p$ Isometric to Subspaces of $\ell_p$},
	volume = {2},
	year = {1998}}

@article{Dor,
	author = {L. Dor},
	journal = {Israel J. Math.},
	number = {3--4},
	pages = {260--268},
	title = {Potentials and isometric embeddings in $L_1$},
	volume = {24},
	year = {1977}}

@article{Schneider,
	author = {R. Schneider},
	journal = {Proc. Amer. Math. Soc.},
	pages = {365--368},
	title = {Zonoids whose polars are zonoids},
	volume = {50},
	year = {1975}}

@article{Schoenberg,
	author = {I.J. Schoenberg},
	journal = {Annals of Mathematics},
	number = {4},
	pages = {811--841},
	title = {Metric Spaces and Completely Monotone Functions},
	volume = {39},
	year = {1938}}

@article{Hardin,
	author = {C.D. Hardin},
	journal = {Indiana Univ. Math. Journal},
	number = {3},
	pages = {449--465},
	title = {Isometries on Subspaces of $L^p$},
	volume = {30},
	year = {1981}}

@article{kol1992,
	author = {Koldobsky, A.},
	journal = {Ann. Inst. Henri Poincare},
	number = {3},
	pages = {335--353},
	title = {Generalized Levy representations of norms and isometric embeddings into $L_p$},
	volume = {28},
	year = {1992}}

@article{weil76,
	author = {Weil, W.},
	journal = {Israel J. Math.},
	pages = {352--367},
	title = {Centrally symmetric convex bodies and distributions},
	volume = {24},
	year = {1976}}
\end{document}